\documentclass{amsart}

\theoremstyle{definition}
\newtheorem{theorem}{Theorem}[section]
\newtheorem{definition}[theorem]{Definition}

\newtheorem{lemma}[theorem]{Lemma}
\newtheorem{proposition}[theorem]{Proposition}
\newtheorem{corollary}[theorem]{Corollary}
\newtheorem{conjecture}[theorem]{Conjecture}

\newtheorem*{remark}{Remark}
\newtheorem*{theoremn}{Theorem}

\usepackage[all]{xy}
\usepackage{amsmath, amsthm,amssymb}
\usepackage{bm}

\newcommand{\RR}[0]{\mathbb{R}}

\newcommand{\ZZ}[0]{\mathbb{Z}}

\newcommand{\ddx}[1]{\tfrac{\partial}{\partial #1}}

\newcommand{\dsum}[0]{\oplus}

\newcommand{\sA}[0]{\mathcal{A}}

\newcommand{\cH}[0]{\mathcal{H}}

\DeclareMathOperator{\End}{End}

\DeclareMathOperator{\tr}{tr}

\DeclareMathOperator{\im}{im}

\DeclareMathOperator{\rk}{rk}

\newcommand{\forms}{\Omega}

\newcommand{\bk}{{\bar k}}
\newcommand{\bkp}{{\overline {k+1}}}

\newcommand{\bz}{{\bar0}}
\newcommand{\bo}{{\bar1}}

\newcommand{\scon}{\mathbb{A}}
\newcommand{\sconn}{\mathbb{A}}

\newcommand{\snorm}{\|}

\newcommand{\szeta}[0]{{\bm \zeta}}
\newcommand{\schi}[0]{{\bm \chi}}
\newcommand{\sdet}[0]{\,\mathbf{det}}
\newcommand{\Str}[0]{\,\mathbf{ Tr}}
\newcommand{\str}[0]{\,\mathbf{ tr}}

\newcommand{\mellin}[2]{\mathcal{M}\left[ #1 ; #2\right]}

\title{Twisted Analytic Torsion and Adiabatic Limits}
\author{Ryan Mickler}
\thanks{mickler.r@husky.neu.edu}
\begin{document}

\maketitle

\begin{abstract}
We study an analogue of the analytic torsion for elliptic complexes that are graded by $\mathbb{Z}_2$, orignally constructed by Mathai and Wu \cite{Mathai:2011}. Motivated by topological T-duality, Bouwknegt and Mathai \cite{Bouwknegt:2004} study the complex of forms on an odd-dimensional manifold equipped with with the \emph{twisted} differential $d_H = d+H$, where $H$ is a closed odd-dimensional form. We show that the Ray-Singer metric on this twisted determinant is equal to the untwisted Ray-Singer metric when the determinant lines are identified using a canonical isomorphism. We also study another analytical invariant of the twisted differential, the derived Euler characteristic $\schi'(d_H)$, as defined by Bismut and Zhang \cite{Bismut:1992}. 
\end{abstract}

\section{Twisted Analytic Torsion}

\subsection{Background}
For an elliptic complex $(E,d)$ over an odd-dimensional compact manifold $M$, one can construct a canonical metric $\| \, \|_{RS} $ on the determinant line $\det H(E,d)$ known as the Ray-Singer (RS) metric. This was first done by Quillen \cite{Quillen:1985a}, building on the work of Ray and Singer \cite{Ray:1971}. To construct this, one introduces a Euclidean metric $g$ to the complex, then defines
\[ \| v \|_{RS} := \rho \cdot |v|_{\ker \Delta}, \quad v \in \det H(E,d)\]
where 
\[ \rho =  \left(\prod_{k=0}^{n} \det{}'(\Delta_k)^{\tfrac{1}{2} k (-1)^k}\right) ,\]
and $\Delta_k = d^*_k d_k + d_k d_{k-1}^*$ is the Laplacian on sections of $E_k$ defined using the metric, $\det{}'$ indicates a zeta-regularized determinant, and $ |\cdot|_{\ker \Delta}$ is the metric on $\det H(E,d)$ induced by an orthonormal basis for $\ker \Delta$. It is somewhat remarkable that this metric is independent of the chosen Euclidean metric $g$. The coefficient $\rho$ in front of the metric $|\cdot|_{\ker \Delta}$ is the analytical object originally studied by Ray and Singer. If we multiply $\rho^{-1}$ by a product of orthonormal harmonic volume forms we get an element $\tau \in \det H(E,d)$, also known as the RS torsion, with the property that $\| \tau \|_{RS} = 1$. The analysis of the RS metric is thus closely tied to the analysis of the torsion element $\tau$, and so we occasionally interchange our discussion between the two. This torsion is defined as an analytical analogue of the classical Reidemeister torsion, which is constructed via a similar formula using simplicial techniques. In fact it was shown independently by Cheeger and M\"uller \cite{Cheeger:1977,Muller:1978} that the analytical and combinatorial metrics coincide. In \cite{Mathai:2011}, Mathai and Wu extend this notion to include operators that only preserve the grading modulo 2 in a way that agrees with the Ray-Singer metric when one considers a $\ZZ$-graded complex modulo 2. The prototypical example of such an operator is the \emph{flux twisted} de Rham operator defined as follows: Let $d$ be the usual de Rham exterior derivative acting on the complex of forms, and let $H$ be a $d$-closed odd (possibly non-homogeneous in degree) form on $M$. The operator is $d_H = d+H$ acting on the $\ZZ_2$-graded (even/odd) complex of differential forms on $M$. This operator clearly satisfies $(d_H)^2 = 0$, and thus we have a $\ZZ_2$-graded ($\pm$) elliptic complex $(C = \forms^{\pm}(M), d_H)$, represented by
\begin{displaymath}
\xymatrix{ 
\forms^+(M)  \ar@/^/[r]^{d_H} & \forms^-(M) \ar@/^/[l]^{d_H} \\
}
\end{displaymath}
In this case, Mathai-Wu define
\[ \|v\|_H = \sdet{'}(d^*_H d_H)^{\frac{1}{2}} \cdot | v |_{\ker \Delta_H} \]
where the boldfaced determinant indicates that we are thinking of this as $\ZZ_2$-graded determinant, i.e $\sdet(A) = \det(A_+) / \det(A_-)$ for operators $A$ preserving the $\ZZ_2$-grading (even operators). We use this notation for other $\ZZ_2$-graded extensions of concepts from basic linear algebra, e.g. $\str(A) = \tr(A_+) - \tr(A_-)$. Mathai-Wu show that the metric of the twisted de Rham operator only depends on the cohomology class of the flux form $[H] \in H^3(M,\RR)$. The purpose of this paper is to compare the metric on the twisted complex to that on the untwisted complex, with the main goal being
\begin{theoremn}[\ref{mainthm}]
There is an isomorphism of determinant lines
\[ \kappa : \det H(E,d) \to \det H(E,d_H) \]
under which the twisted and untwisted RS metrics are identified
\[ \| \cdot \|_{RS} =\kappa^*  \|\cdot \|_H  \]
\end{theoremn}

\subsection{The Zeta Function}
The process of computing zeta-regularized determinants requires that we first compute the zeta function of the operator involved. For the twisted complex, we need to study the partial Laplacians $d_H^*d_H$. These operators, although not elliptic, have well defined zeta functions that have the following Laurent expansion near $s=0$ (c.f. \cite{Mathai:2011,Grubb:2005})
\[ \zeta((d_H^*d_H)_\pm,s) = c_0 + c_1 s + O(s^2) \]
Using this expansion, we use the following symbols for the leading terms in the $s$ expansion of the zeta-function.
\begin{eqnarray*}
\label{zetaexp2}
\szeta(d_H^*d_H,s) &= &\zeta((d_H^*d_H)_+,s) - \zeta((d_H^*d_H)_-,s)\\&=& \schi'(d_H^*d_H) + \szeta'(d_H^*d_H) s + O(s^2) 
\end{eqnarray*}
It can easily be shown (c.f. \cite{Mathai:2011}) that the leading term, $\schi'$, is independent of the metric, so we often just write $\schi'(d_H):=\schi'(d_H^* d_H)$. This term is known as the \emph{derived Euler characteristic} in \cite{Bismut:1992} for the following reason. If the complex $E$ is actually a $\ZZ$-graded complex, then we have the following formula for $\schi'$
\[ \schi'(E,d) = \sum_{k=0}^{n} (-1)^k k b_k \]
where $b_k = \dim H^k(E,d)$ are the Betti numbers of the complex. If we let $p_E(t) = \sum_{k=0}^{n} t^k b_k$ be the Poincar\'e polynomial for $E$, then we have $\schi'(E,d) = p'_E(-1)$, hence the `derived' name. However this simple description only works in the case where our $\ZZ_2$-graded complex is actually a $\ZZ$-graded complex that we have reduced the grading modulo two. No such description exists for the 2-periodic case. This derived characteristic can also be thought of as the zeta-regularized rank for the following reason: If $A \in \End(V)$ is a finite linear map, then we have $\det{}'(t A) = t^{\rk A} \det{}'(A)$, where $\det{}'$ indicates the product of the non-zero eigenvalues. Likewise, in the analytic case, we have
\[ \det{}'( t D ) = t^{\chi{}'(D)} \det{}'(D) \]
Thus the derived characteristic is interesting in its own right. Later, we will produce a formula for the derived characteristic in the case of the twisted de Rham complex.

\section{Adiabatic Limits}

Our goal is to compare the regularized determinant of the Laplacians of the de Rham operator $d$ and its twisted counterpart $d_H$. The main object we will be considering is the following \emph{deformation} of the standard de Rham complex $(\forms(M),d)$. 
\[ D_t =  d +  \sum_{i=1}^{n/2} t^i H_{2i+1} \]
Clearly $D_0 = d$ and $D_1 = d_H$. 
We will analyse the behaviour of this family in two regions. Firstly, in a small neighbourhood of $t=0$, we consider it as a germ of an analytic deformation around $t=0$ and we will use techniques developed by Farber \cite{Farber:1995}. Secondly, we consider it as a continuous family in the region $t \in (0,1]$ and show that it is related here to a family where only the metric is varying, instead of the differential, similar to the situation described in the work of Forman \cite{Forman:1994}. By comparing the behaviour of the zeta function of this family in these two regions as $t\to 0$ we can extract information about the determinants at the endpoints. This type of technique is known in the literature as taking an \emph{adiabatic limit} of the operator $d_H$.

\subsection{Small $t$ behaviour}
The paper of Farber \cite{Farber:1995} determines the \emph{small time} behaviour of the Ray-Singer analytic torsion under a deformation of the elliptic complex. The extension from $\ZZ$-graded to the $\ZZ_2$-graded complexes proceeds easily without complication. We need to know precisely how the spectrum of the family $\Delta_t$ behaves around $t=0$. Farber uses the theory of analytic families due to Kato \cite{Kato}, to show that we can gain a complete understanding of such behaviour for our family $D_t$. Here, we consider the behaviour of the regularized determinant component of the RS metric.

Since the family we are considering is polynomial in $t$, we do not need the full machinery of analytic families as used by Farber, and we refer to their paper for a full discussion of the technicalities in that situation. For polynomial families, such as $Q_t = D_t^* D_t$, the situation is much simpler. We can express one of the results of Kato's deformation theory (c.f. \cite{Kato, Farber:1995}) in the following useful way
\begin{theorem}[Kato]
\label{eigenvaluegerms}
Let $Q_t$ be a holomorphic polynomial family. The spectrum of the operator $Q_t$ (near $t=0$) consists of germs of real analytic functions $\lambda(t)$, each of which belong to one of the following types
\begin{enumerate}
\item $\lambda_n(t) \equiv 0, \text{  for  }  1 \leq n \leq N_0$ (The \emph{stable kernel})
\item$\lambda_n(t) = t^{\nu_n} \bar{\lambda}_n(t), \nu_n \geq 1,  \bar{\lambda}_n(0)\neq 0, \text{  for  } N_0 < n \leq N$ (The \emph{unstable kernel})
\item $ \lambda_n(0) \neq 0, \text{  for  } n > N$
\end{enumerate}
where $N_0, N \in \ZZ$ are both finite.
\end{theorem}
Now that we know that the behaviour of the spectrum of the operator $Q_t$ near $t=0$, we can describe the behaviour of the zeta functions as $t\to 0$. Using the previous result, we can show
$\zeta(Q_t,s) = \zeta_2(Q_t,s) + \zeta_3(Q_t,s) $, where the subscript  $\zeta_i$ indicates that we only consider eigenvalues of type $i$. Since there are only finitely many eigenvalues of type 2 (the unstable kernel), we have
\[ \exp \left( -\zeta_2'(Q_t,0 ) \right)= \prod_{n>N_0}^{N} t^{\nu_n} \bar{\lambda}_n(t) \]
There is a countable number of type 3 eigenvalues, $\lambda(t)$, and we know that the values $\lambda(0)$ capture all the non-zero eigenvalues of $Q_0$.
It can be shown (c.f. \cite{Farber:1995}) that the zeta function for the eigenvalues in the stable kernel is a real analytic function of $t$, and $\lim_{t\to 0} \zeta_3(Q_t,s) = \zeta(Q_0,s)$.
Putting this together, we arrive at
\begin{theorem}[Farber \cite{Farber:1995}]
\label{torsionnearzero}
For the deformations $Q_t$, we have
\[ \det{}'(Q_t) = t^\alpha \theta \det{}'(Q_0) f(t) \]
as $t \to 0$, where
\[ \alpha = \sum_{N_0 < N \leq N} \nu_n, \quad \theta = \prod_{N_0 < N \leq N} \bar\lambda_n(0), \]
$\{t^{\nu_n} \bar\lambda_n(t)\}$ are the type $2$ eigenvalues of the operator $Q_t$, and $f(t)$ is a real analytic function with $f(0) = 1$.
\end{theorem}

Applying this theorem to our superdeterminants, we find
\begin{corollary}
\label{germcomplextpower}
For a deformation of 2-periodic elliptic complex $(C^\infty(E),D_t)$, we have as $t\to 0$
\begin{eqnarray} 
\sdet{}'(D_t^* D_t)  &=&   t^{ \alpha_0 - \alpha_1} \theta_\bz /\theta_\bo \sdet{}'(D_0^* D_0) f(t)
\end{eqnarray}
where $f(t)$ is a real analytic function with $f(0) =1$.
\end{corollary}

Our definition of $\alpha$ here may depend on our choice of metric $g$, however Farber \cite{Farber:1995} shows that $\alpha$ can be computed from a structure known as the parametrized Hodge decomposition, which shown to be independent of the metric chosen. In a later section, we will give a formula for the coefficient $\theta$, in terms of a spectral sequence associated to the deformation.

\subsection{Large $t$ behavior}
We now consider the behaviour of the determinant in the region $0 < t \leq 1$ by using the analysis of Forman, \cite{Forman:1994} who considered adiabatic limits of finite operators. Our goal is to get a formula for the variation $\ddx{t} \szeta{}'(D_t^* D_t) $ as $t\to 0$.
Let $N$ be the grading operator of the de Rham complex and let $\rho_t = t^{N/2}$. It is straightforward to check that \begin{equation}
\label{formanidentity}
t^{1/2} D_t = \rho_t d_H \rho_t^{-1}, \quad t>0
\end{equation}
If we choose a metric $g$, we can compute the adjoint operator 
\[ D_t^* =   d^* +  \sum_{i=1}^{n/2} t^i H_{2i+1}^* \]
We then find, c.f. \cite{Forman:1994}, that the adjoint of the family $D_t$ is proportional to the adjoint of $d_H$ in the scaled metric $g_t = tg$, given by the following relation
\[ t^{1/2} D_t^* = \rho_t d_H^{*_t} \rho_t^{-1}, \quad t>0 \]
where $*_t$ indicates the adjoint taken with respect to the scaled metric $g_t$.
Thus we see that, up to an overall scale of $t^{1/2}$ the operator $D_t$ is similar to $d_H$, and $D_t^*$ is similar to $d_H^{*_t}$.
\begin{corollary}
For the family of elliptic complexes $(C,D_t)$, the derived Euler characteristic is independent of $t$ when $t>0$, i.e.
\begin{equation}
\schi'( D_{t}) =\schi'(  d_H)
\end{equation}
\end{corollary}

\begin{proof}
Because conjugation by an invertible operator doesn't affect the spectrum, we have $\szeta(  d_H^{*_t}  d_H, s) =t^{-s} \szeta( D^*_{t} D_{t}, s)$. We already know that the derived Euler characteristic is independent of the choice of metric, so after evaluating at $s=0$, we are free to take $t=1$ on the left hand side. 
\end{proof}

Note that $g_t$ is not a metric at $t=0$, so it is possible that $\schi'(d_H) \neq \schi'(d)$. One of the main results of this paper is to give a formula for the difference $\schi'( d_H) - \schi'( d)$ (proposition \ref{derivedeuleroffluxtwisted}).
In the following, we use the Mellin transform notation
\[ \mellin{F(z)}{z:s} = \int_0^\infty z^s f(z) d\log z \]
We begin with a useful variation formula.
\begin{lemma}
For the family $D_t$, we have
\[   \ddx{t}\szeta(t D_t^* D_t,s) = s \Gamma(s)^{-1} t^{-1} \mellin{\Str( N (e^{-z \Delta_t} - P_t) )}{z:s}, \]
where $P_t$ is the orthogonal projection onto the kernel of $\Delta_t := D_t^* D_t+ D_t D_t^*$.
\end{lemma}
\begin{proof}
We begin with a formula for the derivative of the scaled differential
\[ \ddx{t} (t^{1/2} D_{t}) = [ N/2t , \rho_t d_H \rho_{t}^{-1} ] = t^{-1/2} [N/2, D_{t} ] \]
and
\[ \ddx{t} (t^{1/2} D_{t}^{*} ) = -t^{-1/2} [N/2 ,D_{t}^*] \]
so we see
 \[ \ddx{t}( t D_{t}^{*} D_{t} ) = (-[N/2, D^*_{t}] D_{t} + D^*_{t} [N/2,D_{t}]) \]
Now, using Duhamel's formula (c.f. \cite{BGV})
\begin{eqnarray*}
\ddx{t} \str( e^{-z t D_{t}^{*} D_{t} }-P_t) &=& - z \str( (-[N/2, D^*_{t}] D_{t} + D^*_{t} [N/2,D_{t}])e^{-z t D_{t}^{*} D_{t} } ) \\
&=&  z \str( N e^{-z t \Delta_t} \Delta_t ) \\
\end{eqnarray*}
So
\begin{eqnarray*}
\ddx{t} \szeta(t D_t^* D_t,s) &=& \Gamma(s)^{-1} \mellin{z \str( N e^{-z t \Delta_t} \Delta_t )}{z:s} \\
&=& \Gamma(s)^{-1} t^{-(s+1)} \mellin{z \str( N e^{-z \Delta_t} \Delta_t )}{z:s} \\
&=& -\Gamma(s)^{-1} t^{-(s+1)} \mellin{z \ddx{z} \str( N (e^{-z \Delta_t} - P_t) )}{z:s} \\
&=& s \Gamma(s)^{-1} t^{-(s+1)} \mellin{\str( N (e^{-z \Delta_t} - P_t) )}{z:s}
\end{eqnarray*}
where we have used the following standard rules for the Mellin transform
\[  \mellin{ f(tz) }{z:s} =  t^{-s} \mellin{f(z)}{z:s} \]
\[  \mellin{ z \ddx{z} f(z) }{z:s} =  -s  \mellin{f(z)}{z:s} \]
and we arrive at the result.
\end{proof}
Using this result, we can calculate the variation of the zeta function
\begin{corollary}
The derivative of the zeta function, $\szeta{}'(D_t^* D_t)$, of the family $D_t$ has the following $t$-dependence
\begin{equation}
\ddx{t} \szeta{}'(D_t^* D_t) =  t^{-1}  \schi{}'(d_H )-t^{-1} \Str( N P_t )
\end{equation}
\end{corollary}
\begin{proof}
Clearly, we have 
\begin{eqnarray*}
 \ddx{t} \szeta(t D_t^* D_t,s) &=&  \ddx{t} (t^{-s} \szeta(D_t^* D_t,s)) \\
 &=& -s t^{-(s+1)}  \szeta(D_t^* D_t,s) + t^{-s} \ddx{t}\szeta(D_t^* D_t,s)
\end{eqnarray*}
and so we arrive at
\[   \ddx{t}\szeta(D_t^* D_t,s) =  s t^{-1}  \szeta(D_t^* D_t,s)+s \Gamma(s)^{-1} t^{-1} \mellin{\Str( N (e^{-z \Delta_t} - P_t) )}{z:s} \]
After taking the derivative w.r.t $s$ and evaluating at $s=0$, the contribution of the Mellin transform of $\Str(N e^{-z \Delta_t})$ yields the index term $\Str(N A_{n/2})$ (c.f. \cite{BGV}) which vanishes because $M$ is odd-dimensional, and thus we arrive at the result.
\end{proof}
There also exists a standard variation formula for the change of the zeta function under a change in the metric, c.f. \cite{Mathai:2011}. The main difference between this variation formula and the usual technique is that the family of metrics $g_t$ degenerates at $t=0$, whereas the family of operators $D_t$ is smooth. 
We can now analyze the behaviour of the term $\Str( N P_{t})$ near $t=0$, where we expect only the type 1 eigenvalues to contribute
\begin{proposition}
\label{nptbehaviour}
As $t\to 0$,
\[ \Str( N P_{t}) = \schi'_0 + O(t)\]
where $\schi'_0 := \Str_{\text{type 1}}( N P_{0})$ is the derived Euler characteristic at $t=0$ of the finite collection of type 1 eigenvectors of $ \Delta_t = D_t^* D_t+ D_tD_t^*$. 
\end{proposition}
\begin{remark}
The complex $(C,D_0 =  d)$ is $\ZZ$-graded, and the usual formula $\schi'_0  = \sum_k (-1)^k k N_k$ applies, where $N_k$ is the number of type 1 eigenvalues of degree $k$.
\end{remark}
\begin{proof}
We need to show that there exists an orthonormal basis $\{\phi_i(t)\}$ for the stable kernel of $\tilde \Delta$ such that each of the forms $\phi_i(0)$ are homogenous in the $\ZZ$-grading, not just the $\ZZ_2$-grading. The existence of such a basis follows quite simply from (\cite{Forman:1994} Thm 6). We describe this in the next section (corollary \ref{zgradingofleadingterms}), once we have introduced the spectral sequence of a deformation. Assuming such a basis exists, then we have
\begin{eqnarray*}
 \Str( N P_t ) &=& \sum_{i} ( \phi_i(t) , (-1)^N N  \phi_i(t) )\\
  &=&  \sum_{i} (-1)^{d_i} d_i ( \phi_i(0) , \phi_i(0) ) + O(t) \\
  &=:& \schi'_0 + O(t)
  \end{eqnarray*} 
where $d_i = \deg \phi_i(0)$.
\end{proof}
Using the above, we arrive at
\begin{proposition}
As $t\to 0$,
\label{dtsmalltbehavior}
\[ \sdet{}'(D_t^* D_t) = t^{\schi'_0-\schi{}'(d_H)} \sdet{}'(d_H^* d_H) f(t) \]
where $f(0) \neq 0$.
\end{proposition}
\begin{proof}
For $0 <t <1$, we have
\begin{eqnarray*}
\szeta{}'(D_t^* D_t) &=& \szeta{}'(D_1^* D_1)-\int_{t}^{1} \left( \ddx{s} \szeta{}'(D_s^* D_s) \right) ds \\
& = & \szeta{}'(d_H^* d_H)-\int_{t}^{1} \left(  s^{-1}  \schi{}'(d_H)-s^{-1} \Str( N P_s ) \right) ds \\
& = & \szeta{}'(d_H^* d_H) +\left(  \schi{}'( d_H)-\schi'_0 \right)\log t + \int_{t}^{1} s^{-1}\left(\Str( N P_s) -\schi'_0\right) ds\\
\end{eqnarray*}
Due to proposition (\ref{nptbehaviour}), we know that $\Str( N P_s) -\schi'_0$ is $O(s)$, so the integral is finite as $t\to 0$. This implies
\begin{eqnarray*}
 \sdet{}'(D_t^* D_t) &=& \exp( - \szeta{}'(D_t^* D_t)) \\
 &=& t^{\schi'_0-  \schi{}'(d_H)} \sdet{}'(d_H^* d_H) e^{-h(t)} 
 \end{eqnarray*}
where
\[  h(t) = \int_{t}^{1}  (\Str( N P_s)- \schi'_0) s^{-1} ds. \]
and the result follows. 
\end{proof}
\subsection{Comparison of the two behaviors} In proposition \ref{dtsmalltbehavior}, we found that the limiting behaviour as $t\to0$ of the zeta function $\sdet{}'(D_t^* D_t)$ is
\[ \sdet{}'(D_t^* D_t) = t^{\schi'_0 -  \schi{}'(d_H)} \sdet{}'(d_H^* d_H) f(t)\]
comparing this with the behaviour computed from the germ complex in proposition \ref{germcomplextpower}, 
\[ \sdet{}'(D_t^* D_t ) \sim t^{\alpha_\bz - \alpha_\bo} (\theta_\bz/ \theta_\bo) \sdet{}'(d^* d )  \] 
we arrive at one of our key results
\begin{theorem}
\label{derivedeuleroffluxtwisted}
The derived Euler characteristic $\schi'(d_H)$ of the flux-twisted de Rham complex is an integer, and is given by the following formula
\begin{eqnarray*}
\schi'(d_H) &=& \schi'_0 - \alpha_\bz + \alpha_\bo
\end{eqnarray*} 
\end{theorem}
We can also immediately read off
\begin{corollary}
The determinants are related by
\[ \sdet{}'(d_H^* d_H) =  \left(\theta_\bz /\theta_\bo\right)\sdet{}'(d^* d)\,{f(0)}  \]
where $\theta_{\bk} :=  \prod_{i}^{} \bar \lambda_{i}(0)$ is the product over the type $2$ eigenvalues $\lambda(t) = t^\nu \tilde \lambda(t)$ of $d_H^{*} d_H $ acting on $\forms^\bk$, and
\[ f(0) =\exp \left(- \int_{0}^{1}  \left(\Str( N P_s)- \schi'_0\right)  s^{-1}ds\right)\]
\end{corollary}
In the next section, we will show that there is an alternative description of the term $\left(\theta_\bz /\theta_\bo\right)$ coming from a spectral sequence associated to the deformation. We will also demonstrate that the `defect' term $f(0)$ in the above formula is identically equal to 1. 

\section{The Adiabatic Spectral Sequence}

Here we describe a method to compute the twisted cohomology $H(C,d_H)$ from the untwisted cohomology $H(C,d)$ provided by a spectral sequence. The idea is to successively approximate the kernel of the operator $D_t$ for small $t$, starting from the kernel of the undeformed operator $D_0=d$. With this tool we can extract information about the eigenvalues of the deformed complex, and we can also gain an insight into how the volume form varies along a deformation. We begin by defining the following filtration on $\forms(M)$.  
\[ \sA_i(M) = \sum_{j\geq i, i\equiv j \bmod 2} \forms^{j}(M) \]
Since $d_H$ only increases form degree, it therefore preserves this filtration, an observation noted in \cite{Rohm:1986}. The spectral sequence we are interested in is the Leray spectral sequence for this filtration (c.f. \cite{McCleary:2001})
\begin{proposition}
There is a spectral sequence, the \emph{adiabatic spectral sequence}, $(E_j^\bk,\partial_j)$, with second page $E_2^\bk = H^\bk(d)$, converging in finitely many pages to $H^\bk(d_H)$.
\end{proposition}
For full details we refer to the original paper where this was used in the setting of analytic torsion \cite{Farber:1995}, here we just provide some relevant facts. The spectral sequence page $E_j^\bk$ can be loosely described as of the leading terms $v(0)$ of forms $v(t) \in \forms^{\bk}(M)[t]$ that vanish to order $t^{2j}$ under the Laplacian $\Delta_t = D_t^*D_t + D_t D_t^*$, modulo certain relations amongst the possible extensions $v(t)$. The identification of $E^\bk_{\infty}$ with $H^\bk(d_H)$ is given by $[v(0)] \mapsto [v(1)]$ for a particular choice of $v(t)$ extending $v(0)$.

It was shown in \cite{Forman:1994} that there is also a natural $\ZZ$-grading on these spaces, which we have made reference to before.
\begin{theorem}[Forman]
\label{zgradingcompatibility}
The differential $\partial_j$ of the spectral sequence is compatible with the $\ZZ$-grading on $E^\bk_j$ given by $E_{j,l}^\bk = E_j^\bk \cap \forms^l(M)$, i.e.
\begin{itemize}
\item[1)] $E_j^\bk = \bigoplus_{\ell\geq 0} E_{j,l}^\bk$
\item[2)] $\partial_j E_{j,l}^{\bk} \subset E^\bkp_{j,\ell+j}$
\item[3)] $\partial_j^* E_{j,l}^{\bk} \subset E^\bkp_{j,\ell-j}$
\end{itemize}
\end{theorem}
The compatibility of the leading terms of this spectral sequence with this $\ZZ$-grading yields an important result that we used in the previous chapter
\begin{corollary}
\label{zgradingofleadingterms}
For the complex, $(\forms(M)[t], D_t )$, the space of harmonic forms has a basis in which every element $v(t) = v_0 + O(t)$ has a leading term $v_0$ that is homogenous in the natural $\ZZ$-grading on forms.
\end{corollary}

Given a finite dimensional chain complex $(V,\partial)$, there is no canonical inclusion of the cohomology as a subcomplex $H(V,\partial) \to V$. One such inclusion is given by Hodge theory, where the complex is equipped with inner products, and the representatives for the cohomology are taken to be the harmonic classes of unit norm. There is, however, a \emph{canonical} isomorphism $\det V \cong \det H(V,\partial)$, known as the Knudsen-Mumford map \cite{Knudsen:1976}. Thus, for each page in the spectral sequence, we obtain canonical isomorphisms
\[ \kappa : \det E_j^\bk \to \det H(E_j^\bk,\partial_j)  = \det E_{j+1}^\bk \]
After composing these maps for each page in the sequence after the second page, we obtain a canonical isomorphism 
\[ \kappa : \det H(C,d) \cong \det H(C,d_H) \]

To see how this isomorphism fits into our analysis of the associated metrics on the determinant lines, we can use the following important theorem
\begin{proposition}[Farber \cite{Farber:1995}]
Under the isomorphism $\kappa$ of determinant lines, we have 
\[ \kappa^* \| \cdot \|_{\det \Delta_H} = ( \theta_\bz/\theta_\bo)^{-1/2} \| \cdot \|_{\det \Delta} \] (\ref{germcomplextpower}).
\end{proposition}
Thus we don't need to worry about the spectral term $ ( \theta_\bz/\theta_\bo)$ coming from the finite number of eigenvalues in the unstable kernel, as it will be absorbed by the use of canonical isomorphism of determinant lines.
Combining this theorem with theorem (\ref{derivedeuleroffluxtwisted}), we find
\begin{corollary}
Under the isomorphism $\kappa$ between the corresponding determinant lines the twisted metric is related to the Ray-Singer metric by the following
\[ \| \cdot \|_{RS} =  \Gamma \, \kappa^* \| \cdot \|_{H} \]
where
\[ \log \Gamma = -\int_{0}^{1}  \left(\Str( N P_s)- \schi'_0(d)\right)  s^{-1}ds\]
and $P_s$ is the projection onto the kernel of $\Delta_s = D_s^* D_s+D_s D_s^*$, and $\schi'_0(d)$ is the derived Euler characteristic of the unstable eigenvalues of $D_t$ at $t=0$.
\end{corollary}

\subsection{Vanishing of the Defect}

We now make the observation that since the twisted metric, the untwisted metric and the map $\kappa$ were independent of the various metrics chosen on the complex, the constant $f(0)$ must also be independent of such a choice. Using this, we can show that the defect vanishes.

\begin{lemma}
$\Gamma = 1$.
\end{lemma}
\begin{proof}
Let $\Delta_t^{g}$ be the Laplacian of $D_t$ in the metric $g$. We have noted before that 
\[ t \Delta_t^{g}= \rho_t \Delta_1^{tg}  \rho_t^{-1}\]
If we consider how the kernel projections act, it is easy to see that $\rho_t P_{\ker \Delta_1^{tg} } \rho_t^{-1} = P_{\ker \Delta_t^g}$. We start with
\begin{equation} \log \Gamma = - \int_{0}^{1}  \left(\Str( N P_{\ker \Delta_s^{g} })- \schi'_0(d)\right)  s^{-1}ds \tag{$*$}
\end{equation}
We can then see that the term $\Str( N P_{\ker \Delta_s^{g} })$ is equal to $\Str( N \rho_t P_{\ker \Delta_1^{sg} } \rho_t^{-1})$, and since $[N,\rho^t] = 0$, this is equal to $\Str( N P_{\ker \Delta_1^{sg} })$. With this, we have
\begin{equation*}
\log \Gamma = -\int_{0}^{1} \left( \Str( N P_{\ker \Delta_1^{sg} })- \schi'_0(d) \right) s^{-1} ds
\end{equation*}
However, since we know that this quantity must be independent of $g$, we are free to scale $g$ by $\lambda \in \RR_{>0}$. 
\begin{eqnarray*}
\log \Gamma = -\int_{0}^{1} \left( \Str( N P_{\ker \Delta_1^{s\lambda^{-1}g} })- \schi'_0(d) \right) s^{-1} ds
\end{eqnarray*}
Now we change coordinates $s = \lambda^{-1} s$, and we find
\begin{eqnarray*}
\log \Gamma = -\int_{0}^{\lambda} \left( \Str( N P_{\ker \Delta_1^{sg} })- \schi'_0(d) \right) s^{-1} ds
\end{eqnarray*}
We can now change the integrand back to its original form, 
\begin{eqnarray*}
\log \Gamma = -\int_{0}^{\lambda}  \left(\Str( N P_{\ker \Delta_s^{g} })- \schi'_0(d)\right)  s^{-1}ds
\end{eqnarray*}
Comparing this expression with $(*)$, we see that the independence of $\log \Gamma$ on $\lambda$ forces it to vanish.
\end{proof}
With this, we can state the main result of this work
\begin{theorem}
\label{mainthm}
The twisted and untwisted metrics are identified under the canonical isomorphism of determinant lines.
\[  \| \cdot \|_{RS} = \kappa^* \| \cdot \|_H \]
\end{theorem}
This theorem allows us to compute the regularized determinant of the Laplacian $d_H^*d_H$ in several examples, and provides extensions of any properties of the untwisted metric to the twisted case.

\section{Examples}
We can use our main results to easily calculate some of the invariants we are interested in for some simple examples.
\begin{lemma}
\label{twistcohomvanish}
If $H(C,d_H) = 0$, then 
\[ \schi'( d_H) = -\alpha_\bz + \alpha_\bo\]
\end{lemma}
\begin{proof}
The twisted cohomology groups are isomorphic for all $t>0$, so for them to vanish at $t=1$ implies that the projection $P_t \equiv 0$ for $t>0$ and so $\schi'_0(d)=0$.
\end{proof}

In \cite{Mathai:2011}, the case where $H$ is of top degree was studied
\begin{proposition}[\cite{Mathai:2011}]
If $M^n$ is an odd-dimensional, compact, oriented Riemannian manifold, $H$ is a multiple of the volume form,
\[ \sdet{}'(d_H^* d_H) =  \snorm H \snorm^{2b_0} \sdet{}'(d^* d) \]where $b_0 = \dim H^0(M,\RR)$ is the $0$th Betti number
\end{proposition}
In their paper, Mathai-Wu prove this result using a subtle factorization argument of regularized determinants developed by Kontsevich and Vishik in \cite{kontsevich:1995a}. Here, we reprove this theorem without the need for such powerful machinery, and answer a question posed in their paper about the value of the derived Euler characteristic.
\begin{theorem}
If $M^n$, is an odd-dimensional, compact, oriented Riemannian manifold of dimension $n>1$ and $H$ is a multiple of the volume form, then
\[ \sdet{}'(d_H^* d_H) =  \snorm H \snorm^{2b_0} \sdet{}'(d^* d) \]
and furthermore
\[ \schi'(d_H) = \schi'(d) + b_0\]
where $b_0$ is the $0$th Betti number.
\end{theorem}
\begin{proof}
We wish to determine the unstable eigenvectors and eigenvalues of the Laplacian. Let $s = t^{(n-1)/2}$. For $t > 0$, the operator $D_t = d+s H$, although it is $\ZZ_2$-graded, preserves the $\ZZ$-grading on forms except in the first two and last two degrees. $D_t: C^0 \dsum C^{n-1} \to C^1 \dsum C^n$ has kernel $C^{n-1}_{cl}$, since $H \wedge : C^0 \to C^n$ is a bijection. Now $D_t=d$ when acting on forms of non-zero degree, so the $D_t$-closed forms are just $d$-closed forms in this range. The $D_t$-exact forms are given by $d$-exact forms of degree $>1$, as well as the composite forms $(d \tau, s H\wedge \tau) \in C^1 \dsum C^n, \tau \in C^0$. So the cohomology groups of $d_H$ are the same as those for $d$, except possibly in degrees $0,1$ and $n$. Consider a closed form $\beta \in C^n_{cl}$. Since $H^n(C,d) = \RR^{b_0}[H]$, choose a form $\gamma \in C^{n-1}$ so that $\beta-d\gamma = \lambda s H, \lambda \in C^0_{cl}$, i.e $\beta = D_t (\lambda + \gamma)$, thus $\beta$ is $D_t$-exact.  Now, take a closed form $\alpha \in C^1_{cl}$, which is necessarily cohomologous to $(\alpha+d\tau,s H\wedge \tau)$. Since $s H\wedge \tau$ is closed, $\alpha$ is also exact by the previous argument. Thus $\alpha \sim \alpha + d\tau$, and if $\alpha$ is $d$-harmonic, then it is also $D_t$-harmonic. Thus we find
\[ H(C,D_t) \cong \bigoplus_{k=1}^{n-1} H^k(C,d) \]
We can also see this by noting that $H(C,d_H) = H(H(C,d),H\wedge)$, but it is not obvious that we can choose homogenous harmonic representatives for our cohomology classes.
This means that the $D_t$-harmonic forms are $\ZZ$-graded, and so we have 
\[ \Str( N P_t) = \schi'_0(d) = \sum_{k=1}^{n-1} (-1)^k k b_k = \schi'(d) + n b_n = \schi'(d)+n b_0\]
We need to look at the behavior of the unstable kernel under this deformation. The locally constant functions have a basis $\{v_i\}_{i=1}^{b_0}$, where $v_i$ is non-vanishing only on the $i$th connected component. These trivially satisfy $d^* d v_i = 0$, and we find $D_t^* D_t v_i = s^2 \snorm H \snorm^2 v_i $ since $H$ is co-closed. This tells us each $v_i$ is an eigenvector of $D_t^* D_t$,  with eigenvalue $\lambda(t) = s^2 \snorm H \snorm^2 = t^{n-1}  \snorm H \snorm^2$, and these are the only even forms in the unstable kernel. Thus we find
\[ \theta_0 = \snorm H \snorm^{2 b_0}, \quad \alpha_\bz = (n-1) b_0, \quad \alpha_\bo = 0 \]
where the $\alpha_\bk$ are defined as in theorem \ref{torsionnearzero}.
So using theorem \ref{derivedeuleroffluxtwisted}, we find
\begin{eqnarray*}
 \schi'(d_H) &=& \schi'_0(d) - \alpha_\bz + \alpha_\bo \\
 &=& \schi'(d) + n b_0 - (n-1) b_0  \\
  &=& \schi'(d) + b_0 
 \end{eqnarray*}
Furthermore
\begin{eqnarray*}
\sdet{}'(d_H^* d_H) &=& (\theta_\bz/\theta_\bo)  \sdet{}'(d^* d) \\
&=&  \snorm H \snorm^{2b_0}   \sdet{}'(d^* d)
 \end{eqnarray*}
\end{proof}
The main difference between our proof and that of Mathai-Wu is that we only need to focus on the contributions of the finite collection of unstable eigenvalues, as suggested by Farber's work, and not the entire spectrum of the twisted differential. 

\section{Relation to $T$-duality} 

One motivation for the work of Mathai-Wu was to study the behaviour of the analytic torsion under a transformation known as topological $T$-duality. This concept comes from a geometric duality between the spacetimes of type IIA and IIB string theories. After discarding all the geometric information and all of the string theory, one is left with the following topological set up. Let $p_1: P_1\to M$ be a principal circle bundle over an even dimensional base $M$, and $H_1 \in H^3(P_1,\ZZ)$. The we can pushforward $H_1$ along $p_1$ to get $(p_1)_*H_1 \in H^2(M,\ZZ)$. This class on the base represents the first Chern class of another principal circle bundle $p_2: P_2 \to M$, i.e. $c_1(P_2) = (p_1)_*H_1$. The question is: can we find a degree three cohomology class on $P_2$ such that this picture becomes symmetric?
\begin{proposition}[Topological T-duality \cite{Bouwknegt:2004}]
There exists a unique class $H_2 \in H^3(P_2,\ZZ)$ such that $(p_2)_*H_2 = c_1(P_1)$ and that satisfies $p_1^*H_2 = p_2^*H_1 \in H^3(P_1 \times_M P_2,\ZZ)$.
\end{proposition}
The general philosophy of this kind of T-duality is that we can use a Fourier-Mukai type transform (which we won't describe here) to relate certain 'twisted' $\ZZ_2$-graded objects between $P_1$ and $P_2$. This is loosely summarized as
\[
\begin{array}{c}
\{\text{  even/odd } H_1 \text{-twisted objects on }P_1 \}  \\
\Updownarrow\\
\{\text{  odd/even }H_2 \text{-twisted objects on }P_2 \} \\
\end{array}
\]
The canonical example is given by
\begin{proposition}[\cite{Bouwknegt:2004}]
With the data described earlier, there is a $T$-duality isomorphism between the twisted cohomology groups
\[ \mathcal{T} : H^\pm(P_1,d_{H_1}) \to H^{\mp}(P_2,d_{H_2}) \]
Here $H_i$ are taken to be certain de Rham representatives of the corresponding classes in cohomology. 
\end{proposition}

Thus we also get an induced map \[\det \mathcal{T} : \det H(P_1,d_{H_1}) \to \det H(P_2,d_{H_2})^* \]
The natural question in our context is how this isomorphism relates the corresponding analytic torsion elements.
\begin{proposition}[Mathai-Wu \cite{Mathai:2011}]
Under the $T$-duality map, we have
\[ \det \mathcal{T} : 2^{\chi(M)} \tau( P_1, d_{H_1} ) \mapsto \left( 2^{\chi(M)} \tau( P_2,d_{H_2} ) \right)^{-1}\]
\end{proposition}
Thus the T-duality map preserves the analytic torsion. Now we can combine this theory with our theorem comparing twisted to untwisted analytic torsion in the following way: Let $P_1$ is a non-trivial circle bundle over $M$, and suppose we want to compute the untwisted torsion $\tau(P_1,d)$. We can first compute the $T$-dual circle bundle $P_2$. We have chosen $H_1=0$, so the dual circle bundle will have vanishing Chern class, thus $P_2 \cong M \times S^1$. The corresponding dual flux must then be $H_2 = c_1(P_1) \boxtimes [S^1] \in H^3(M \times S^1,\ZZ)$. We now apply the T-duality map $\det \mathcal{T}$, which tells us that we only need to compute $\tau(M \times S^1, d_{H_2})$. We now use the canonical isomorphism $\kappa$ to show that we only need to compute $\tau(M \times S^1,d)$. This torsion is significantly simpler, since the non-trivial bundle $P_1$ has been replaced by a trivial bundle. Thus in this way, we can use some \emph{cohomological} maps between determinant lines to simplify the calculation of the untwisted torsion.\section{Extension to Flat Superconnections}
One of the main sources for $\ZZ_2$-graded complexes comes from flat superconnections. These operators, first defined by Quillen \cite{Quillen:1985} are a generalization of the concept of a connection on a vector bundle to that of super-bundles, or $\ZZ_2$-graded vector bundles.
\begin{definition} A \emph{superconnection} on a $\ZZ_2$-graded vector bundle $E = E^+ \ominus E^-$ on $M$, is an odd parity, first order differential operator $\sconn \in \End^-( \sA(M,E) )$, satisfying the Leibnitz rule 
\[ [\scon,\alpha] = d\alpha \quad \text{for } \alpha \in \sA(M, E), \]
where we think of a form $\alpha$ as operating on the complex of forms by exterior multiplication.
\end{definition}
If we consider $L = M \times\RR$ as the trivial superbundle, $E = E^+ \ominus E^- = L \ominus 0$, then $d_H$ is a superconnection on this bundle.

Note that if we fold up a $\ZZ$-graded complex with a connection, we get a superconnection.
We can extend $\scon$ to act on $\sA(M,\End E)$ by $\scon \alpha := [\scon,\alpha]$ for $\alpha \in \sA(M,\End E )$. Using these definitions, we find that the \emph{curvature} $F_\scon := \scon^2$ is given by exterior multiplication by a form, $F_\scon \in \sA^+(M,\End E)$. If we decompose the connection into homogenous components $\scon = \sum_{i=0}^{n} \scon_{[i]}$, where $\scon_{[i]} : \Gamma(M,E) \to \sA^i(M,E)$, we see that $\scon_{[1]}$ is a connection on the bundle $E$ that preserves the $\ZZ_2$-grading, and that for each $i\neq 1$, $\scon_{[i]}$ is given by exterior multiplication with some form $\scon_{[i]} \alpha = \omega_{[i]} \wedge \alpha$, where $\omega_{[i]} \in \sA^{i,-}(M,\End  E ) $.
We say a superconnection is \emph{flat} if its curvature vanishes identically, i.e. $F_\scon = 0$. Note that if a superconnection $\scon$ is flat, it does not necessarily mean that the connection component $\scon_{[1]}$ is also flat. Given a flat superconnection, $\scon$, we can form the $\ZZ_2$-graded complex $(\sA^\pm(M,E) ,\scon)$. The Leibnitz property ensures that this complex is elliptic, and so following Mathai-Wu \cite{Mathai:2011a}, we can define the RS metric $\| \cdot \|_{(E,\scon)}$ on $\det H(E,\scon)$.
Since the connection $\nabla := \scon_{[1]}$ is not flat, there is not any torsion invariant associated to $\nabla$. However, if we label $\partial := \scon_{[0]},\phi := \scon_{[2]}$, and examine the first three flatness constraints (homogeneous components of $F_\scon =0$) for $\scon$, we find
\[ \partial^2=0,\qquad  [\partial,\nabla] =0\qquad  [\partial,\phi] + F_\nabla =0\]
We see that the first condition implies that the bundle map $d$ squares to zero, i.e. $(E, \partial)$ is a $\ZZ_2$-graded complex of vector bundles. The second flatness condition implies that the connection $\nabla$ commutes with $\partial$. The third condition implies that the curvature $F_\nabla = \nabla^2$, is $\partial$-chain homotopically trivial, i.e. vanishes in the $\partial$-cohomology. Now, if we assume that $\ker \partial \subset E$ is a constant-rank vector bundle, then $\im \partial \subset \ker \partial$ is a constant-rank vector sub-bundle, and so the quotient vector bundle $\cH := H(E,\partial) = \ker \partial/ \im \partial$ is also constant rank.

Thus if the bundle $\cH$ is suitably well defined, then the connection $\nabla$ descends to a \emph{flat} superconnection $\tilde\nabla$ on the quotient, known as the \emph{Gauss-Manin} connection. We are then able to construct the RS metric $\| \cdot \|_{(\cH,\tilde\nabla)}$ on $\det H(\cH,\tilde\nabla)$.
So we find that from a flat superconnection we can construct two metrics: the RS metric of the complex $(\scon)$, and the RS metric of the Gauss-Manin connection $(\cH,\tilde \nabla)$.

There is also a canonical filtration on the complex of forms valued in a super-bundle. Define the filtration by form degree
\[ \sA^{i,\bk}(M,E) = \bigoplus_{j\geq i} \forms^j(M,E^{\bk-j}) \]
It is clear that these spaces filter $\sA^{\bk}(M,E)$
\[ \sA^{\bk}(M,E) = \sA^{0,\bk}(M,E) \supset \sA^{1,\bk}(M,E) \supset \ldots \supset \sA^{n,\bk}(M,E) \]
and that $\scon : \sA^{i,\bk}(M,E) \to \sA^{i,\bkp}(M,E)$. Thus, as before, we can construct a spectral sequence, for which we find $E_2^\bk = H^\bk(\cH,\tilde\nabla)$ and that converges to $E_\infty^\bk= H^\bk(E,\scon)$. From this we construct a identification of determinant lines $\kappa : \det H(\cH,\tilde\nabla) \to \det H( E,\scon)$.
\begin{conjecture}
Under the identification of determinant lines, we have
\[  \| \cdot \|_{(\cH,\tilde\nabla)} \mapsto \kappa^* \| \cdot \|_{(E,\scon)}\]
\end{conjecture}

Thus we have shown this conjecture holds in the case where $E = E^+ = \RR$, and $\scon = d_H$ is an arbitrary superconnection on this bundle. This case was easy to handle because $\scon_{[0]} = \partial= 0$, and so $H(E,\partial) = E$. For the general case, a more sophisticated analysis of the eigenvalues of the family associated to the scaled operator $\scon_t = \rho_t \scon \rho_t^{-1}$ should be involved.

\section{Acknowledgements}
This work encompasses part of the content of author's master's thesis completed while studying at the University of Adelaide with the support of an Australian Postgraduate Award (APA). The author benefited greatly from supervision by Varghese Mathai and Michael Murray, and from conversations with Maxim Braverman.

\renewcommand{\bibname}{References}

\bibliographystyle{plain}

\bibliography{/Users/ryanmickler/Documents/Archive/MasterArchive}

\end{document}